\documentclass[a4paper,12pt,reqno]{amsart}
\usepackage{latexsym}
\usepackage{amssymb} 
\usepackage{mathrsfs}
\usepackage{amsmath}
\usepackage{latexsym}
\usepackage{delarray}
\usepackage{amssymb,amsmath,amsfonts,amsthm,mathrsfs}

\usepackage[colorlinks]{hyperref}
\renewcommand\eqref[1]{(\ref{#1})} 

 \newtheorem{thm}{Theorem}[section]
 \newtheorem{cor}[thm]{Corollary}
 \newtheorem{lem}[thm]{Lemma}
 \newtheorem{prop}[thm]{Proposition}
 \theoremstyle{definition}
 \newtheorem{defn}[thm]{Definition}
 \theoremstyle{remark}
 \newtheorem{rem}[thm]{Remark}
 \newtheorem{ex}[thm]{Example}
 \numberwithin{equation}{section}
\newcommand{\half}{\frac{1}{2}}

\newcommand{\ene}{\mathbb{N}}

\newcommand{\er}{\mathbb{R}}
\newcommand{\ce}{\mathbb{C}}

\newcommand{\arn}{{\mathbb{R}}^n}

\newcommand{\efee}{\mathcal{F}}
\newcommand{\efel}{\mathcal{F}_L}
\newcommand{\efela}{\mathcal{F}_{L^*}}

\newcommand{\bi}{\begin{itemize}}
\newcommand{\bba}{\mathcal{B}}
\newcommand{\ei}{\end{itemize}}
\newcommand{\be}{\begin{enumerate}}
\newcommand{\ee}{\end{enumerate}}
\newcommand{\beq}{\begin{equation}}
\newcommand{\eq}{\end{equation}}

\newcommand{\oM}{\mathring{M}}

\newcommand{\cm}{C_{L^*}^{\infty}(M)}
\newcommand{\cmo}{C_{L_0}^{\infty}(M)}
\newcommand{\cml}{C_{L}^{\infty}(M)}

\newcommand{\mI}{\mathcal{I}}

\def\jp#1{{\left\langle{#1}\right\rangle}}

\DeclareMathOperator{\Tr}{Tr}

\def\HS{{\mathtt{HS}}}

\def\Tn{{{\mathbb T}^n}}
\def\Zn{{{\mathbb Z}^n}}

\def\SU2{{{\rm SU(2)}}}
\def\SO3{{{\rm SO(3)}}}
\def\lapsu2{{{\mathcal L}_\SU2}}

\def\jp{{\langle\xi\rangle}}

\def\Dom{{{\rm Dom}}}

\begin{document}
\title[Schatten classes and Nuclearity on manifolds with boundary]
{Schatten classes, nuclearity and nonharmonic analysis on compact manifolds with boundary}
\author[Julio Delgado]{Julio Delgado}


\address{Department of Mathematics\\
Imperial College London\\
180 Queen's Gate, London SW7 2AZ\\
United Kingdom}

\email{j.delgado@imperial.ac.uk}

\thanks{The first author was supported by the
Leverhulme Research Grant RPG-2014-02.
The second author was supported by
 the EPSRC Grant EP/K039407/1.
 No new data was collected or generated during the course of the research.}

\author{Michael Ruzhansky}

\address{Department of Mathematics\\
Imperial College London\\
180 Queen's Gate, London SW7 2AZ\\
United Kingdom}

\email{m.ruzhansky@imperial.ac.uk}


\author{Niyaz Tokmagambetov}

\address{Al-Farabi Kazakh National University\\
71 Al-Farabi Ave., Almaty, 050040\\
Kazakhstan}

\email{niyaz.tokmagambetov@gmail.com}



\date{\today}

\begin{abstract}
Given a compact manifold $M$ with boundary $\partial M$, in this paper we 
introduce a global symbolic calculus of pseudo-differential operators associated to $(M,\partial M)$. 
The symbols of operators with boundary conditions on $\partial M$
are defined in terms of the biorthogonal expansions in eigenfunctions of a fixed
operator $L$ with the same boundary conditions on $\partial M$.
The boundary $\partial M$ is allowed to have (arbitrary) singularities.
 As an application, several criteria for the membership in Schatten classes on $L^2(M)$ and $r$-nuclearity on $L^p(M)$ are obtained. We also describe a new addition to the Grothendieck-Lidskii formula in this setting. Examples and applications are given to operators on $M=[0,1]^n$ with non-periodic boundary conditions,
and of operators with non-local boundary conditions.
\end{abstract}

\maketitle
\section{Introduction}

If $A$ is an operator on a manifold $M$ with boundary, satisfying some boundary conditions on
$\partial M$, we establish criteria for $A$ to belong to $r$-Schatten classes on $L^2(M)$
and to be $r$-nuclear on $L^p(M)$. 
Our analysis is carried out by applying a version of the global Fourier analysis on $M$ expressed
in terms of another (model) operator $L$ on $M$ with the same domain as that of $A$ 
(or, in other words, with the same boundary conditions on $\partial M$). 
This also extends to the setting of manifolds with boundary the notions of the
pseudo-differential analysis of boundary value problems in the Euclidean space
as developed in \cite{rto:nhs}. 

In general the lack of some symmetry in the boundary conditions gives rise
 to a non-selfadjoint setting and therefore when
  considering eigenfunction expansions one is led to consider biorthogonal systems rather than orthogonal ones. Thus the concept
  of Riesz basis becomes crucial for us.  The Riesz bases have been intensively studied, see for instance   
  \cite{bdhk:rb}, \cite{rba:int} and references therein for recent interest directions.

The study of differential operators on manifolds with boundary is an active field of research, a few very recent examples  
include \cite{jlr:car}, \cite{jlr:sp}, \cite{gmk:bv}, \cite{gmk:bvf},   and the reader can find many other references in those works. 
There are different well known approaches to the problem of defining a pseudo-differential calculus on manifolds 
with boundary, see, for example, \cite{bm:se1,maz:bi, mm:pse1, ss:bv, ee:ps} and references therein.
 The main difference  of our approach from all those above is that our symbols
are globally defined and that we do not assume any regularity for the boundary $\partial M$. 

Still, it becomes possible to draw rather general conclusions. In particular, we will not require that the operator $L$ is elliptic nor
that it is self-adjoint. Moreover, in the presented general framework the smoothness of the boundary is not required, and the exact form
of the boundary conditions is also not essential. 
Moreover, we do not assume regularity of the symbols to ensure Schatten properties of operators.
This latter feature is however an advantage of the approach relying on the global symbolic analysis as was already
demonstrated in authors' work in the context of compact Lie groups \cite{dr13a:nuclp} or of compact
manifolds \cite{dr:sdk,dr:suffkernel}.

Perhaps the simplest illustrating example of the setting in the present paper would be
the set $M=[0,1]^n$ where we impose on operators 
non-periodic boundary conditions of the type
$h_j f(x)|_{x_j=0}=f(x)|_{x_j=1}$, for some collection $h_j>0$, $j=1,\ldots,n$. 
This extends the periodic case when $h_j=1$ for all $j=1,\ldots,n$, in which case the global toroidal
pseudo-differential calculus as developed in \cite{Ruzhansky-Turunen-JFAA-torus,rt:book} can be applied.
However, the latter (periodic) setting is considerably simpler because such calculus is based on the self-adjoint
operator in $M$ (the Laplacian) with periodic boundary conditions, in particular leading to the orthonormal basis 
of its eigenfunctions. Due to the lack of such self-adjointness in the non-periodic problem, our present analysis
is based on a basis in $L^2(M)$ which is not orthonormal but which respects the given boundary 
conditions. As such, the subsequent analysis resembles more that in Banach spaces (e.g. in $L^p$-spaces) than
the classical one in Hilbert spaces. From this point of view, Grothendieck's notion of $r$-nuclearity conveniently replaces
the $r$-Schatten classes.  

Indeed, when studying Schatten classes for operators directly from the definition, the problem of 
 understanding the compositions $A A^*$ or $A^*A$ arises. With an appropriate definition of Fourier multipliers in our context
 given in the sequel, we note that even in the simplest case of a $L$-Fourier multiplier $A$, 
 the operator $A^*$ is a $L^*$-Fourier multiplier, but not necessarily a $L$-Fourier multiplier. 
 Moreover, in general, the operators $A$ and $A^*$ satisfy different conditions on the
 boundary.
 Therefore, unless the model operator $L_M$ ($L$ equipped with conditions on $\partial M$) is 
 self-adjoint, we can not compose $A$ and $A^*$ on their domains.
 These observations explain why the method of studying Schatten classes
 via the notion of $r$-nuclearity on Banach spaces becomes more appropriate especially in this case
 of non-self-adjoint boundary value problems. If the operators are continuously extendible to
 $L^2(M)$ they can still be composed and the notions of $r$-Schatten classes and $r$-nuclearity coincide on $L^2(M)$.

\smallskip
 To formulate the notions more precisely, let $L$ be a pseudo-differential operator of order $m$ on the interior $\oM$ of $M$. 
 This interior $\oM$ will be a smooth manifold without boundary and the standard theory of pseudo-differential operators applies there.
 We assume that some boundary conditions called (BC) are fixed and lead to a discrete spectrum with a family of eigenfunctions yielding a biorthogonal basis in $L^2(M)$. However, it is important to point out that
 the operator $L$ does not have to be self-adjoint nor elliptic. 
 For a discussion on general biorthogonal systems we refer the reader to Bari \cite{bari}
 (see also Gelfand \cite{Gelfand:on-Bari}), and now we formulate our
 assumptions precisely. The discrete sets of eigenvalues and eigenfunctions will be indexed by a countable set $\mI$. 
We consider the spectrum $\{\lambda_{\xi}\in\ce:\xi\in\mI\}$ of $L$ with corresponding eigenfunctions in $L^2(M)$ denoted by $u_{\xi}$, i.e.
\begin{equation}\label{EQ:LL}
Lu_{\xi}=\lambda_{\xi}u_{\xi} \mbox{ in }\oM,\, \mbox{ for all }\xi\in\mI,
\end{equation}
and the eigenfunctions $u_{\xi}$ satisfy the boundary conditions (BC). The conjugate spectral problem is
\[L^*v_{\xi}=\overline{\lambda_{\xi}}v_{\xi} \mbox{ in }\oM,\, \mbox{ for all }\xi\in\mI,\]
which we equip with the conjugate boundary conditions $\mathrm{(BC)}^*$. We assume that the functions $u_{\xi}, v_{\xi}$ are normalised, i.e.
 $\|u_{\xi}\|_{L^2}=\|v_{\xi}\|_{L^2}=1$ for all $\xi\in \mI$. Moreover, we can take biorthogonal systems $\{u_{\xi}\}_{\xi\in\mI}$
and $\{v_{\xi}\}_{\xi\in\mI}$, i.e. 
\beq\label{bio} (u_{\xi}, v_{\eta})_{L^2}=0 \mbox{ for }\xi\neq\eta, \,\mbox{ and }\, (u_{\xi}, v_{\eta})_{L^2}=1 \mbox{ for }\xi=\eta,\eq
where
\[(f,g)_{L^2}=\int\limits_Mf(x)\overline{g(x)}dx\]
is the usual inner product of the Hilbert space $L^2(M)$. We also assume that
 the system $\{u_{\xi}\}$ is a basis of $L^2(M)$, i.e. for every $f\in L^2(M)$ there exists a unique series $\sum_{\xi\in\mI}a_{\xi}u_{\xi}$ that converges to $f$ in $L^2(M)$. It is well known that (cf. \cite{bari}) the system $\{u_{\xi}\}$ is a basis of $L^2(M)$ if and only if the system $\{v_{\xi}\}$ is a basis of $L^2(M)$.

By associating a discrete Fourier analysis to the system $\{u_{\xi}\}$, we can introduce a full symbol for a given operator
 acting on suitable functions over $M$ as an extension of the setting established in \cite{rto:nhs}. We will describe the basic elements of such symbols in Section 
\ref{SEC:review}. 

On the other hand, we will also employ the concept of nuclearity on Banach spaces. 
Here, although $L^2(M)$ is a Hilbert space, the non-orthonormality of the basis $\{u_\xi\}$ makes the analysis
more reminiscent of that in Banach spaces.

Let $\bba_1, \bba_2$ be Banach spaces and let $0<r\leq 1$. A linear operator $T$
from $\bba_1$ into $\bba_2$ is called {\em r-nuclear} if there exist sequences
$(x_{n}^{\prime})\mbox{ in } \bba_1' $ and $(y_n) \mbox{ in } \bba_2$ so that
\beq 
Tx= \sum\limits_{n=1}^{\infty} \left <x,x_{n}'\right>y_n \,\mbox{ and }\,
\sum\limits_{n=1}^{\infty} \|x_{n}'\|^{r}_{\bba_1'}\|y_n\|^{r}_{\bba_2} < \infty.\label{rn}
\eq
We associate a quasi-norm $n_r(T)$ by
\beq\label{qs1}
n_r(T):=\inf\{\left(\sum\limits_{n=1}^{\infty} \|x_{n}'\|^{r}_{\bba_1'}\|y_n\|^{r}_{\bba_2}\right)^{\frac{1}{r}}\},
\eq
where the infimum is taken over the  representations of $T$ as in \eqref{rn}. 
When $r=1$ the $1$-nuclear operators agree with 
the class of nuclear operators, and in that case this definition also agrees with the concept of trace class operators
 in the setting of Hilbert spaces ($\bba_1=\bba_2=H$). More generally, Oloff proved in \cite{Oloff:pnorm} that the class of $r$-nuclear
operators coincides with the Schatten class $S_{r}(H)$ when $\bba_1=\bba_2=H$ is a Hilbert space and 
$0<r\leq 1$. Moreover, Oloff proved that 
\beq\label{olo1}\|T\|_{S_r}=n_r(T),\eq
where $\|\cdot\|_{S_r}$ denotes the classical Schatten quasi-norms in terms of singular values.

 If $T:\bba\rightarrow\bba$ is nuclear, it is natural to
 attempt to define its  trace by
\begin{equation}\label{EQ:Trace}
\Tr (T):=\sum\limits_{n=1}^{\infty}x_{n}'(y_n),
\end{equation}
where $T=\sum\limits_{n=1}^{\infty}x_{n}'\otimes y_n$ is a
representation of $T$ as in \eqref{rn}. Grothendieck \cite{gro:me} proved that the trace $\Tr(T)$ is well defined for 
 all nuclear operators $T$ on $\bba$ if and only if the Banach space
 $\bba$ has the {\em approximation property}
  (see also Pietsch \cite{piet:book}),  which means that for every compact
   set $K$ in $\bba$ and for every $\epsilon >0$ there exists $F\in \mathcal{F}(\bba) $ such that
\[\|x-Fx\|_{\bba}<\epsilon\quad \textrm{ for all } x\in K,\]
where we have denoted by $\mathcal{F}(\bba)$ the space of all finite rank bounded linear operators
 on $\bba$. 

As we know from Lidskii \cite{li:formula}, for trace class operators 
in Hilbert spaces the operator trace is equal to the
sum of the eigenvalues of the operator counted with multiplicities. This property is nowadays called the
Lidskii formula.
An important feature on Banach  spaces even endowed with the approximation property is that the Lidskii formula does not hold in general for nuclear operators. 
 In \cite{gro:me} Grothendieck proved that if $T$ is $\frac 23$-nuclear from $\bba$ into $\bba$ for a Banach space $\bba$, then
\beq\Tr(T)=\sum\limits_{j=1}^{\infty}\lambda_j(T),\label{lia1}\eq
where $\lambda_j(T)\,\, (j=1,2,\dots)$ are the eigenvalues of $T$ with multiplicities taken into account,
and $\Tr(T)$ is as in \eqref{EQ:Trace}. We will refer to \eqref{lia1} as the Grothendieck-Lidskii formula.

 Grothendieck established applications to the distribution of eigenvalues of operators
in Banach spaces. We refer to \cite{dr13a:nuclp} for several conclusions 
in the setting of compact Lie groups
concerning
summability and distribution of eigenvalues of operators on $L^{p}$-spaces once
we have information on their $r$-nuclearity. In the particular case of $L^p$ spaces, the summability of eigenvalues found by Grothendieck
 can be improved. Indeed, in \cite{jkmr:glp} Johnson, K\"onig, Maurey and Retherford  proved the following theorem,
 which will give us an application on the distribution of $L^p$-eigenvalues in terms of the nuclear or Schatten index
 of the operator class:

\begin{thm}[\cite{jkmr:glp}]\label{THM:Johnson}
Let $0<r\leq 1$ and $1\leq p\leq \infty$. Let $(\Omega, \mu)$ be a measure space and $\frac{1}{s}=\frac{1}{r}-|\frac 12-\frac{1}{p}|$.  
If $T$ is a $r$-nuclear operator on $L^p(\mu)$, then
\beq\label{disteig}
\left(\sum\limits_{j=1}^{\infty}|\lambda_j(T)|^s\right)^{\frac{1}{s}}\leq C_rn_{r}(T),
\eq
where $C_r$ only depends on $r$.
\end{thm}

Kernel conditions on compact manifolds for Schatten classes have been investigated in 
\cite{dr:sdk, dr:suffkernel}.  Schatten properties on modulation spaces have been studied
in \cite{Toft:modul,Toft:modul2,Toft:Schatten-modulation-2008}, however, also in such low-regularity
situations certain regularity of symbols has to be assumed. This is not the case with our approach
when in most criteria only the $L^p$-integrability of the appearing symbols is usually assumed. 

In Remark \ref{REM1} we explain how the case of spectral multipliers (in $L$ or in $L^*$) is included in our setting.

\section{Nonharmonic analysis and global symbols}
\label{SEC:review}

We introduce in this section the basic (global) symbolic calculus on manifolds with
boundary based on biorthogonal systems as an extension of the theory developed in \cite{rto:nhs}. We adapt it to the present situation of a manifold with boundary. While in \cite{rto:nhs} only domains in the Euclidean space $\mathbb R^n$ were considered, the proofs easily extend to the present setting of general manifolds with boundary. Thus, in this section we record elements of the analysis from \cite{rto:nhs} in the present setting but we omit the proofs if they are a straightforward extension of those in \cite{rto:nhs}.

\smallskip
Following the terminology of Paley and Wiener \cite{Paley-Wiener:book-1934}, such biorthogonal analysis can be
thought of as part of nonharmonic analysis, see \cite{rto:nhs} for a more extensive discussion.

\medskip
Henceforth $M$ denotes a compact manifold of dimension $n$ with boundary $\partial M$ and $\mathring{M}$ the interior of $M$. We will also denote by $L_{M}$ the boundary value problem determined by the pseudo-differential operator $L$ of order $m$ on $\mathring{M}$ equipped with the boundary conditions (BC) on $\partial M$.
This can be also thought of in an abstract way of having an operator $L$ with some domain in $L^2(M)$.

\medskip
To the problem $L_{M}$ with $L$ of order $m$ we associate the weight 
\beq\label{we1}\langle\xi\rangle:=(1+|\lambda_{\xi}|^2)^{\frac{1}{2m}}\simeq (1+|\lambda_{\xi}|)^{\frac{1}{m}},
\eq
with $\lambda_{\xi}$ as in \eqref{EQ:LL}.
We also assume that there exists a number $s_0\in\er$ such that
\beq\label{we2}\sum\limits_{\xi\in\mI}\langle\xi\rangle^{-s_0}<\infty .\eq
It is natural to expect that e.g. if $L$ is elliptic one can take any $s_0>n$.

Throughout this paper we 
will also assume the following technical condition to ensure that the eigenfunctions of $L$ and $L^*$ do not have zeros.
 In particular such property will allow us to obtain suitable formulae for the symbol of an operator.
\begin{defn} The system $\{u_{\xi}:\xi\in\mI\}$ is called a WZ-system 
(WZ stands for `without zeros')
if the functions $u_{\xi}, v_{\xi}$ do not have zeros in 
 $M$ for all $\xi\in\ene_0$, and if there exists $C>0 $ and $N\geq 0$ such that
\begin{align*} 
\inf\limits_{x\in M}|u_{\xi}(x)|\geq & C\langle\xi\rangle ^{-N},\\
\inf\limits_{x\in M}|v_{\xi}(x)|\geq & C\langle\xi\rangle ^{-N},
\end{align*}
as $\langle\xi\rangle\rightarrow\infty.$
\end{defn}

\begin{rem}
The reader can find examples and a discussion of WZ-systems in Section 2 of \cite{rto:nhs} for 
 the case $M=\overline{\Omega}$ where $\Omega$ is an open bounded subset of $\arn$.
There are plenty of problems where this conditions holds, a few of them are described
in Section \ref{SEC:example}. In principle, this assumption can be removed,
however, this leads to a considerably more technical exposition since the analysis
becomes `matrix-valued': we can group eigenfunctions into a vector so that its
elements do not all vanish at the same time - this leads to a matrix-valued version of
the symbolic analysis. A typical example of such situation is of operators on compact
Lie groups or on compact homogeneous spaces: in this case the eigenfunctions of the Laplacian spanning a given eigenspace
do not vanish at the same time, see the analysis developed in \cite{rt:book}.
The group structure can be removed: see \cite{dr14a:fsymbsch} for the case of invariant
operators on compact manifolds. However, since in this paper we are dealing
with biorthogonal systems instead of orthonormal bases, we choose to deal
with scalar symbols and WZ-condition for the simplicity in the exposition of ideas.
In subsequent work this assumption will be removed.
\end{rem}

\begin{defn} The space $C_L^{\infty}(M):=\Dom(L^{\infty})$ is called the space of test functions for $L_{M}$. We define
\[\Dom(L^{\infty}):=\bigcap\limits_{k=1}^{\infty}\Dom(L^k),\]
where $\Dom(L^k)$ is the domain of $L^k$, defined as
\[\Dom(L^k):=\{f\in L^2(M):L^j f\in L^2(M),\, j=0,1,2,\dots, k\},\]
so that the boundary conditions (BC) are satisfied by all the operators $L^k$. The Fr\'echet topology
 of $C_L^{\infty}(M)$ is given by the family of norms
\beq\label{top1ab}\|\varphi\|_{C_L^{k}}:=\max\limits_{j\leq k}\|L^j\varphi\|_{L^2(M)},\,\, k\in\ene_0,\, \varphi\in C_L^{\infty}(M).\eq
 \end{defn}
In an analogous way, we define $C_{L^*}^{\infty}(M)$ corresponding to the adjoint $L_{M}^*$ by 
\[C_{L^*}^{\infty}(M):=\Dom((L^*)^{\infty})=\bigcap\limits_{k=1}^{\infty}\Dom((L^*)^{k}),\]
where $\Dom((L^*)^{k})$ is the domain of the operator $(L^*)^{k}$,
\[\Dom((L^*)^{k}):=\{f\in L^2(M):(L^*)^j f\in L^2(M),\, j=0,1,2,\dots, k\},\]
which then also has to satisfy the adjoint boundary conditions corresponding to the operator $L_{M}^*$. The Fr\'echet topology
 of $C_{L^*}^{\infty}(M)$ is given by the family of norms
\beq\label{to2}\|\psi\|_{C_{L^*}^{k}}:=\max\limits_{j\leq k}\|(L^*)^j\psi\|_{L^2(M)},\,\, k\in\ene_0,\, \psi\in C_{L^*}^{\infty}(M).\eq

Since here we are assuming that the system $\{u_{\xi}\}$ is a basis of $L^2(M)$, it follows that  $C_L^{\infty}(M)$ and $C_{L^*}^{\infty}(M)$ are dense in $L^2(M)$.

If $L_M$ is self-adjoint, i.e. if $L_M^*=L_M$ with equality of domains, then $C_{L^*}^{\infty}(M)=C_{L}^{\infty}(M)$.

The $L^2$-duality for a general pair $f\in C_{L}^{\infty}(M), g\in C_{L^*}^{\infty}(M)$ makes sense in view of the formula
\beq\label{dua1a} (Lf,g)_{L^2(M)}=(f,L^*g)_{L^2(M)}.\eq
Therefore in view of the formula \eqref{dua1a} it makes sense to define the distributions  $\mathcal{D}_{L}'(M)$ as the space which 
is dual to $C_{L^*}^{\infty}(M)$. Note that the respective boundary conditions of $L_M$ and $L_M^*$ are satisfied by the choice of
 $f$ and $g$ in corresponding domains.

\begin{defn} The space $\mathcal{D}_{L}'(M):=\mathcal{L}(C_{L^*}^{\infty}(M),\ce)$ of linear continuous functionals on $C_{L^*}^{\infty}(M)$ is called the space of $L$-distributions. The continuity can be understood either in terms of the topology \eqref{to2} or in terms of sequences as in Proposition \ref{ch1da}. For $w\in \mathcal{D}_{L}'(M)$ and $\varphi\in C_{L^*}^{\infty}(M)$, we shall write
\[w(\varphi)=\langle w,\varphi\rangle .\]
For any $\psi\in C_{L}^{\infty}(M)$, 
\[C_{L^*}^{\infty}(M)\ni\varphi\mapsto\int\limits_M\psi(x)\varphi(x)dx\]
is an $L$-distribution, which gives an embedding $\psi\in C_{L}^{\infty}(M)\hookrightarrow\mathcal{D}_{L}'(M)$. We observe that in the distributional notation formula
 \eqref{dua1a} becomes 
\beq\label{ldf1} \langle L\psi,\varphi\rangle =\langle\psi,\overline{L^*\overline{\varphi}}\rangle .\eq
\end{defn}

With the topology on $C_{L}^{\infty}(M)$ defined by \eqref{top1ab}, the space 
\[\mathcal{D}_{L^*}'(M):=\mathcal{L}(C_{L}^{\infty}(M),\ce)\]
 of linear continuous functionals on $C_{L}^{\infty}(M)$ is called the space of $L^*$-distributions.

The following proposition characterises the distributions in $\mathcal{D}_{L}'(M)$.
\begin{prop}\label{ch1da} A linear functional $w$ on $C_{L^*}^{\infty}(M)$ belongs to $\mathcal{D}_{L}'(M)$ if and only if there exists
 a constant $C>0$ and $k\in\ene_0$ such that
\beq\label{wia} |w(\varphi)|\leq C\|\varphi\|_{C_{L^*}^{k}}\,\mbox{ for all }\varphi\in C_{L^*}^{\infty}(M).\eq
\end{prop}

The space $\mathcal{D}_{L}'(M)$ has many similarities with the usual spaces of distributions. For example, suppose that for a linear continuous operator $D:C_{L}^{\infty}(M)\rightarrow C_{L}^{\infty}(M)$ its adjoint $D^*$ preserves the adjoint  boundary conditions (domain) of $L_M^*$ and is continuous on the space $C_{L^*}^{\infty}(M)$, i.e. that the operator $D^*:C_{L^*}^{\infty}(M)\rightarrow C_{L^*}^{\infty}(M)$ is continuous. Then we can extend
$D$ to $\mathcal{D}_{L}'(M)$ by
\[\langle Dw,\varphi\rangle :=\langle w,\overline{D^*\overline{\varphi}}\rangle\quad (w\in\mathcal{D}_{L}'(M),\, \varphi\in C_{L^*}^{\infty}(M)) .\]
This extends \eqref{dua1a} from $L$ to other operators. The convergence
 in the linear space $\mathcal{D}_{L}'(M)$ is the usual weak convergence with respect to the space $C_{L^*}^{\infty}(M) $. The following principle of uniform boundedness is based on the Banach-Steinhaus Theorem applied to the Fr\'echet space $C_{L^*}^{\infty}(M)$.

\begin{lem}\label{wsq1} Let $\{w_j\}_{j\in\ene}$ be a sequence in $\mathcal{D}_{L}'(M)$ with the property that for every $\varphi\in C_{L^*}^{\infty}(M)$, the sequence  $\{w_j(\varphi)\}_{j\in\ene}$ is bounded in $\ce$. Then there exist constants $c>0$ and $k\in\ene_0$ such that
\beq\label{weq2} |w_j(\varphi)|\leq c\|\varphi\|_{C_{L^*}^{k}} \,\,\mbox{ for all }\,j\in\ene,\, \varphi\in C_{L^*}^{\infty}(M).\eq
\end{lem}
The lemma above leads to the following property of completeness of the space of $L$-distributions.

\begin{thm}\label{twdr} Let $\{w_j\}_{j\in\ene}$ be a sequence in $\mathcal{D}_{L}'(M)$ with the property that for every $\varphi\in C_{L^*}^{\infty}(M)$, the sequence  $\{w_j(\varphi)\}_{j\in\ene}$ converges in $\ce$ as $j\rightarrow\infty.$ Denote the limit by $w(\varphi)$.\\

$(i)$ Then $w:\varphi\mapsto w(\varphi)$ defines an $L$-distribution on $M$. Furthermore,
\[\lim\limits_{j\rightarrow\infty}w_j=w \,\mbox{ in }\, \mathcal{D}_{L}'(M).\]
$(ii)$ If $\varphi_j\rightarrow\varphi$ in $C_{L^*}^{\infty}(M)$, then
\[\lim\limits_{j\rightarrow\infty}w_j(\varphi_j)=w(\varphi) \,\mbox{ in }\, \ce.\]
\end{thm}

We now associate a Fourier transform to the operator $L_{M}$ and establish its main properties. The particular feature compared with the self-adjoint case is that if $L_{M}$ is not self-adjoint we have to be sure of the choice of the right functions from the available biorthogonal families of $u_{\xi}$ and $v_{\xi}$. We start by defining the spaces we require for the study of the Fourier transform.

From now on, we assume that the boundary conditions are closed under taking limits  in the strong
uniform topology to ensure that the strongly convergent series preserve the boundary conditions. More precisely,

\[{\rm (BC+)} \mbox{ Assume that, with }L_0 \mbox{ denoting }L \,{\rm or }\, L^*,\, \mbox{ if } f_j\in\cmo \mbox{ satisfies }\]
\[f_j\rightarrow f \mbox{ in }\cmo, \mbox{ then }\,f\in\cmo.\]

Let $\mathcal{S}(\mI)$ denote the space of rapidly decaying functions $\varphi:\mI\rightarrow \ce$. That is, $\varphi\in\mathcal{S}(\mI)$ if for every $\ell<\infty$ there exists a constant $C_{\varphi, \ell}$ such that
\[|\varphi(\xi)|\leq C_{\varphi, \ell}\langle\xi\rangle^{-\ell}\]
for all $\xi\in\mI$. With the corresponding topology, we note that $\jp$ is already adapted to our boundary value problem
 since it is defined by \eqref{we1}.
 
The topology on $\mathcal{S}(\mI)$ is given by the seminorms $p_k$, where $k\in\ene_0$ and 
\[p_k(\varphi):=\sup\limits_{\xi\in\mI}\jp^k|\varphi(\xi)|.\]
Continuous linear functionals on $\mathcal{S}(\mI)$ are of the form $\varphi\mapsto\langle u,\varphi\rangle\:=\sum\limits_{\xi\in\mI}u(\xi)\varphi(\xi),$
where functions $u:\mI\rightarrow \ce$ grow at most polynomially at infinity, i.e. there exist constants $\ell<\infty$ and $C_{u,\ell}$ such that
\[|u(\xi)|\leq C_{u,\ell}\jp^{\ell}\]
for all $\xi\in\mI$. Such distributions $u$ form the space of distributions which we denote by $\mathcal{S}'(\mI)$.
 
\subsection{Fourier transform} 
We can now define the $L$-Fourier transform on $C_L^{\infty}(M)$ and study its main properties.
\begin{defn}\label{fl1} We define the $L$-Fourier transform
\[(\efee_Lf)(\xi)=(f\mapsto\widehat{f}):\cml\rightarrow\mathcal{S}(\mI)\]
by
\beq\label{foul1}\widehat{f}(\xi):=(\efee_Lf)(\xi)=\int\limits_{M}f(x)\overline{v_{\xi}(x)}dx.\eq

Analogously, we define the $L^*$-Fourier transform
\[(\efela f)(\xi)=(f\mapsto\widehat{f_{\ast}}):\cm\rightarrow\mathcal{S}(\mI)\]
by
\beq\label{foul2}\widehat{f_{\ast}}(\xi):=(\efela f)(\xi)=\int\limits_{M}f(x)\overline{u_{\xi}(x)}dx.\eq
\end{defn}

The expresions \eqref{foul1} and \eqref{foul2} are well-defined by the Cauchy-Schwarz inequality,
 indeed,
 \beq\label{fouiq}|\widehat{f}(\xi)|=\left|\int\limits_{M}f(x)\overline{v_{\xi}(x)}dx\right |\leq \|f\|_{L^2}\|v_{\xi}\|_{L^2}=\|f\|_{L^2}<\infty .\eq
Moreover, we have

\begin{prop} \label{finv} The $L$-Fourier transform $\efel$ is a bijective homeomorphism from $\cml$ into $\mathcal{S}(\mI)$. Its inverse $\efel^{-1}:\mathcal{S}(\mI)\rightarrow\cml$ is given by
\beq\label{invfo1}(\efel^{-1}h)(x)=\sum\limits_{\xi\in\mI}h(\xi)u_{\xi}(x),\,\, h\in\mathcal{S}(\mI),\eq
so that the Fourier inversion formula is given by 
\beq\label{eq:finv}f(x)=\sum\limits_{\xi\in\mI}\widehat{f}(\xi)u_{\xi}(x),\,\, f\in \cml.\eq
Similarly, $\efela:\cm\rightarrow\mathcal{S}(\mI) $ is a bijective homeomorphism and its inverse $\efela^{-1}:\mathcal{S}(\mI) \rightarrow \cm$ is given by
\beq\label{finv1q}(\efela^{-1}h)(x)=\sum\limits_{\xi\in\mI}h(\xi)v_{\xi}(x),\,\, h\in\mathcal{S}(\mI),\eq
so that the conjugate  Fourier inversion formula is given by 
\beq\label{eq:finv3} f(x)=\sum\limits_{\xi\in\mI}\widehat{f_{\ast}}(\xi)v_{\xi}(x),\,\, f\in \cm.\eq
\end{prop}

By dualising the inverse $L$-Fourier transform $\efel^{-1}:\mathcal{S}(\mI)\rightarrow\cml$, the $L$-Fourier transform extends uniquely to the mapping
 \[\efel:\mathcal{D}_L'(M)\rightarrow\mathcal{S}'(\mI),\] 
by the formula 
\beq\label{fl1i2}\langle\efel w,\varphi\rangle:=\langle w,\overline{\efela^{-1}\overline{\varphi}}\rangle,\,\, \mbox{ with }\, w\in\mathcal{D}_L'(M),\,\varphi\in\mathcal{S}(\mI).\eq
It can be readily seen that if $w\in\mathcal{D}_L'(M)$ then $\widehat{w}\in\mathcal{S}'(\mI)$. The reason for taking complex conjugates in \eqref{fl1i2} is that, if $w\in\cml$, we have the equality
\begin{align*} \langle\widehat{w},\varphi\rangle &=\sum\limits_{\xi\in\mI}\widehat{w}(\xi)\varphi(\xi)=\sum\limits_{\xi\in\mI}\left(\int\limits_{M}w(x)\overline{v_{\xi}(x)}dx\right)\varphi(\xi)\\
&=\int\limits_{M}w(x)\overline{\left(\sum\limits_{\xi\in\mI}\overline{\varphi(\xi)}v_{\xi}(x)\right)}dx=\int\limits_{M}w(x)\overline{(\efela^{-1}\overline{\varphi})(x)}dx=\langle w,\overline{\efela^{-1}\overline{\varphi}}\rangle.
\end{align*}
Analogously, we have the mapping
\[\efela:\mathcal{D}_{L^*}'(M)\rightarrow \mathcal{S}'(\mI)\]
defined by the formula
\beq\label{fl1i2a}\langle\efela w,\varphi\rangle:=\langle w,\overline{\efel^{-1}\overline{\varphi}}\rangle,\,\, \mbox{ with }\, w\in\mathcal{D}_{L^*}'(M),\,\varphi\in\mathcal{S}(\mI).\eq
It can be also seen that if $w\in\mathcal{D}_{L^*}'(M)$ then $\widehat{w}\in \mathcal{S}'(\mI)$.

We note that since systems of $u_{\xi}$ and of $v_{\xi}$ are Riesz bases, we can also compare $L^2$-norms of functions with sums of squares of Fourier coefficients. The following statement follows form the work of Bari \cite{bari} (Theorem 9):
\begin{lem}\label{lnqer} There exist constants $k, K,m,\ell >0$ such that for every $f\in L^2(M)$ we have
\[m^2\|f\|_{L^2}^2\leq \sum\limits_{\xi\in\mI}|\widehat{f}(\xi)|^2\leq \ell^2\|f\|_{L^2}^2\]
and
\[k^2\|f\|_{L^2}^2\leq \sum\limits_{\xi\in\mI}|\widehat{f_{\ast}}(\xi)|^2\leq K^2\|f\|_{L^2}^2.\]
\end{lem}
\subsection{Convolution}
We now adapt a convolution to the problem $L_{M}$.
\begin{defn}\label{conv1c} For $f,g\in \cml$ we define their $L$-convolution by
\beq\label{conv12 }(f\ast_L g)(x):=\sum\limits_{\xi\in\mI}\widehat{f}(\xi)\widehat{g}(\xi)u_{\xi}(x).\eq
\end{defn}
By Proposition \ref{invfo1} it is well-defined and we have $f\ast_L g\in\cml$.

Moreover, due to the rapid decay of the $L$-Fourier coefficients of functions in $\cml$ compared to a fixed polynomial growth of elements of $ \mathcal{S}'(\mI)$, the Definition \ref{conv1c} still makes sense if $f\in\mathcal{D}_{L}'(M)$ and $g\in\cml$, with $f\ast_L g\in\cml$. 

Analogously to the $L$-convolution, we can introduce the $L^*$-convolution. Thus, for $f,g\in\cm$, we define the $L^*$-convolution using the $L^*$-Fourier transform by
\beq\label{conv12a }(f\widetilde{\ast}_L g)(x):=\sum\limits_{\xi\in\mI}\widehat{f_{\ast}}(\xi)\widehat{g_{\ast}}(\xi)v_{\xi}(x).\eq
Its properties are similar to those of the $L$-convolution, so we may formulate and prove only the former.

\begin{prop} For any $f, g\in\cml$ we have
\[\widehat{f\ast_L g}=\widehat{f}\,\widehat{g}.\]
The convolution is commutative and associative. If $g\in\cml$, then for all $f\in\mathcal{D}_{L}'(M)$ we have
\beq\label{conv12w}f\ast_L g\in\cml.\eq
If $f,g\in L^2(M)$, then $f\ast_L g\in L^1(M)$ with
\[\|f\ast_L g\|_{L^1}\leq C|M|^{\half}\|f\|_{L^2}\|g\|_{L^2},\]
where $|M|$ denotes the volume of $M$, and $C$ is independent of $f, g, M$.
\end{prop}
\begin{proof} We have
\begin{align*} \efel(f\ast_L g)(\xi)&=\int_M\sum\limits_{\xi\in\mI}\widehat{f}(\eta)\widehat{g}(\eta)u_{\eta}(x)\overline{v_{\xi}(x)}dx\\
&=\sum\limits_{\xi\in\mI}\widehat{f}(\eta)\widehat{g}(\eta)\int_M\sum\limits_{\xi\in\mI}u_{\eta}(x)\overline{v_{\xi}(x)}dx\\
&=\widehat{f}(\xi)\widehat{g}(\xi).
\end{align*}
This also implies the commutativity of the convolution in view of the bijectivity
of the Fourier transform. The associativity follows from this as well from the associativity of the
multiplication on the Fourier transform side.
In order to prove \eqref{conv12w}, we observe that
\[L^k(f\ast_L g)(x)=\sum\limits_{\xi\in\mI}\widehat{f}(\xi)\widehat{g}(\xi)\lambda_{\xi}^ku_{\xi}(x),\]
and the series converges absolutely since $\widehat{g}\in\mathcal{S}(\mI)$. 
By (BC+), the boundary conditions are also satisfied by the sum. This shows that
  $f\ast_L g\in\cml$.
  
  For the last statement, a simple calculation gives us
\begin{align*} \int\limits_M|(f\ast_L g)(x)dx &\leq \int\limits_M \sum\limits_{\xi\in\mI}|\widehat{f}(\xi)\widehat{g}(\xi)||u_{\xi}(x)|dx\\
&\leq \sum\limits_{\xi\in\mI}|\widehat{f}(\xi)\widehat{g}(\xi)|\|u_{\xi}\|_{L^1}\\
&\leq C\|f\|_{L^2}\|g\|_{L^2}\sup\limits_{\xi\in\mI}\|u_{\xi}\|_{L^1},
\end{align*} 
the latter inequality is obtained by Lemma \ref{lnqer}. Since $M$ is bounded, by the Cauchy-Schwarz inequality we have
\[\|u_{\xi}\|_{L^1}\leq |M|^{\half}\|u_{\xi}\|_{L^2}=|M|^{\half}\]
for all $\xi\in\mI$, where $|M|$ is the volume of $M$. This concludes the proof.
\end{proof}

\subsection{ $L$- symbols and $L$-Fourier multipliers}

The Schwartz integral kernel theorem holds in this setting in analogy to 
 \cite[Section 8]{rto:nhs}, and so
for a linear operator $A:\cml\rightarrow\mathcal{D}_{L}'(M)$, the corresponding convolution kernel $k_A(x)\in\mathcal{D}_{L}'(M)$ is determined by 
\[Af(x)=(k_A(x)\ast f)(x).\]

We now associate the notion of a global symbol associated to the operator $A$ with respect
to $L_{M}$ and also to its
adjoint. 

\begin{defn}\label{Lsymbol} The $L$-symbol of a continuous linear operator $A:\cml\rightarrow\cml$
 at $x\in M$ and $\xi\in\mI$ is defined by 
\[\sigma_{A}(x,\xi):=\widehat{k_A(x)}(\xi)=\efel(k_A(x))(\xi).\]
\end{defn}
The following theorem furnishes a representation of an operator in terms of its $L$-symbol and a formula for the $L$-symbol  in terms of the operator and the biorthogonal system. The proof is analogous to the one of Theorem 9.2  in \cite{rto:nhs} and we omit it.
\begin{thm} Let  $A:\cml\rightarrow\cml$ be a continuous linear operator with $L$-symbol $\sigma_A$. Then
\beq Af(x)=\sum\limits_{\xi\in\mI}u_{\xi}(x)\sigma_A(x,\xi)\widehat{f}(\xi)\eq
for every $f\in\cml$ and $x\in M$. The $L$-symbol $\sigma_A$ satisfies 
\beq\label{sy1a}\sigma_A(x,\xi)=u_{\xi}(x)^{-1}(Au_{\xi})(x)\eq
for all $x\in M$ and $\xi\in\mI$.
\end{thm}

As a consequence, one can deduce the following formula for the kernel of $A$ in terms of the $L$-symbol and it will be crucial for the analysis in this work:
\beq\label{ker1}
K_A(x,y)=\sum\limits_{\xi\in\mI}u_{\xi}(x)\sigma_A(x,\xi)\overline{v_{\xi}(y)}.
\eq

One can also associate a notion of multipliers to the $L$-Fourier transform. 
\begin{defn}\label{Lfm} Let $A:\cml\rightarrow\cml$ be a continuous linear operator. We will say
 that $A$ is a $L$-Fourier multiplier if it satisfies
\[
\efel (Af)(\xi)=\sigma(\xi)\efel (f)(\xi),\, f\in\cml,\textrm{ for all }\xi\in\mI,
\] 
for some $\sigma:\mI\rightarrow\ce$.
\end{defn}
Analogously we define $L^*$-Fourier multipliers, a notion which will naturally appear in the study of adjoints (see Proposition \ref{admu}).
\begin{defn}\label{Lfm2} Let $B:\cm\rightarrow\cm$ be a continuous linear operator. We will say
 that $B$ is a $L^*$-Fourier multiplier if it satisfies
\[
\efela (Bf)(\xi)=\tau(\xi)\efela (f)(\xi),\, f\in\cm,\textrm{ for all }\xi\in\mI,
\]
for some $\tau:\mI\rightarrow\ce$.
\end{defn}

\begin{rem}\label{REM1}
We note that due to the formula \eqref{sy1a} for symbols, we have that if $\sigma_A(x,\xi)=\sigma_A(\xi)$ does not depend on $x$, then we have $Au_\xi=\sigma_A(\xi)u_\xi$, so that 
$\sigma_A(\xi)$ are the eigenvalues of $A$ corresponding to the eigenfunctions $u_\xi$. Recalling that $\lambda_\xi$ are the corresponding eigenvalues of the operator $L$, see \eqref{EQ:LL}, if 
$\lambda_\xi$'s are distinct and 
$\phi$ is a function taking $\lambda_\xi$'s to $\sigma_A(\xi)$, then we can also regard $A$ as the spectral multiplier $A=\phi(L)$.
This is not the case when we allow the symbol to take multiplicities into account, as in 
\cite{dr14a:fsymbsch}, but where only the self-adjoint model operators $L$ on manifolds without boundaries were considered. In view of such further developments possible also in the present setting, we still prefer to use the term $L$-Fourier multiplier here, in line with \cite{dr14a:fsymbsch} and \cite{rto:nhs}.
\end{rem}

\subsection{Nuclearity on Lebesgue spaces}

As a final preliminary, we record the following characterisation of $r$-nuclear operators which is a consequence of results in 
\cite{del:tracetop}. In the statement below we shall consider $({\Omega}_1,{\mathcal{M}}_1,\mu_1 )$ and
$({\Omega}_2,{\mathcal{M}}_2,{\mu}_2)$ to be two $\sigma$-finite measure spaces. 

\begin{thm}\label{ch2} 
Let $0\leq r <1$ and $1\leq p_1,p_2 <\infty$ with $q_1$ such that
$\frac{1}{p_1}+\frac{1}{q_1}=1$. 
 An operator $T:L^{p_1}({\mu}_1)\rightarrow L^{p_2}({\mu}_2)$ is $r$-nuclear if and only if 
 there exist sequences
 $(g_k)_k$ in $L^{p_2}({\mu}_2)$, and $(h_k)_k$ in $L^{q_1}(\mu_1)$ such that
  $\sum \limits_{k=1}^\infty \| g_k\|_{L^{p_2}}^r
 \|h_k\|_{L^{q_1}}^r<\infty$, and such that for all $f\in L^{p_1}(\mu_1)$ we have
$$
Tf(x)=\int\left(\sum\limits_{k=1}^{\infty}
  g_k(x)h_k(y)\right)f(y)d\mu_1(y), \quad \mbox{for  a.e }x.$$
\end{thm}
In our particular setting we will consider $M=\Omega_1=\Omega_2,$ $\mathcal{M}=\mathcal{M}_1=\mathcal{M}_2$ the $\sigma$-Borel algebra, and $dx=\mu_1=\mu_2$ a positive measure on $M$. 

\section{Schatten classes and nuclearity}
In this section we establish the main results of this work. 
 We start by studying the $r$-nuclearity on $L^p(M)$ spaces and we also formulate some consequences related to 
 the Grothendieck-Lidskii formula for the trace. From now on we will be considering the biorthogonal system $u_{\xi}, v_{\xi}$ associated to the problem $L_M$ as basic blocks in the decomposition of kernels.
\begin{thm}\label{sp1} Let $0<r\leq 1$ and $1\leq p_1,p_2 <\infty$ with $q_1$ such that
$\frac{1}{p_1}+\frac{1}{q_1}=1$. Let $A:\cml\rightarrow \cml $ be a continuous linear operator with $L$-symbol $\sigma_A$ such that
\beq\label{inc5}\sum\limits_{\xi\in\mI}\|\sigma_A(\cdot,\xi)u_{\xi}(\cdot)\|_{L^{p_2}}^r\|v_{\xi}\|_{L^{q_1}}^r<\infty.\eq
Then $A$ is $r$-nuclear from $L^{p_1}(M)$ into $L^{p_2}(M)$.

In particular, if $p=p_1=p_2$ and \eqref{inc5} holds, then 
\beq\Tr(A)=\sum\limits_{\xi\in\mI}\int\limits_{M}u_{\xi}(x)\sigma_{A}(x,\xi)\overline{v_{\xi}(x)}dx,\label{trsglb}\eq
and
\beq\label{disteig2}\left(\sum\limits_{j=1}^{\infty}|\lambda_j(A)|^s\right)^{\frac{1}{s}}\leq C_r\left(\sum\limits_{\xi\in\mI}\|\sigma_A(\cdot,\xi)u_{\xi}(\cdot)\|_{L^{p}}^r\|v_{\xi}\|_{L^{q}}^r\right)^{\frac 1r},\eq
where $\,\frac{1}{s}=\frac{1}{r}-|\frac 12-\frac{1}{p}|$ and $C_r$ is a constant depending only on $r$. Moreover, if 
$0<r\leq \frac 23$ the following formula for the trace holds:
\beq \Tr(A)=\sum\limits_{\xi\in\mI}\int\limits_{M}u_{\xi}(x)\sigma_{A}(x,\xi)\overline{v_{\xi}(x)}dx=\sum\limits_{j=1}^{\infty}\lambda_j(A),\label{gltr1}\eq
where $\lambda_j(A)$, $(j=1,2,\dots)$ are the eigenvalues of $A$ with multiplicities taken into account.
\end{thm}

\begin{proof} 
We recall from \eqref{ker1} that the kernel $K_A(x,y)$ of $A$ is of the form
 \beq\label{cond1an}K_A(x,y)=\sum\limits_{\xi\in\mI}u_{\xi}(x)\sigma_A(x,\xi)\overline{v_{\xi}(y)}.\eq
We take $g_{\xi}(x):=u_{\xi}(x)\sigma_A(x,\xi),\,\, h_{\xi}(y):=\overline{v_{\xi}(y)}$.

\smallskip
Then we have
\[\sum\limits_{\xi\in\mI}\|g_{\xi}\|_{L^{p_2}}^r\|h_{\xi}\|_{L^{q_1}}^r=\sum\limits_{\xi\in\mI}\|\sigma_A(\cdot,\xi)u_{\xi}(\cdot)\|_{L^{p_2}}^r\|v_{\xi}\|_{L^{q_1}}^r<\infty.\]
The $r$-nuclearity of $A$ now follows from Theorem \ref{ch2}.

By using the expression of  $K_A$ in terms of $g_{\xi}, h_{\xi}$, the fact that $L^p$ satisfies the approximation property and the definition of the trace \eqref{EQ:Trace} we obtain
\beq\Tr(A)=\sum\limits_{\xi\in\mI}\int\limits_{M}g_{\xi}(x)h_{\xi}(x)dx=\sum\limits_{\xi\in\mI}\int\limits_{M}u_{\xi}(x)\sigma_{A}(x,\xi)\overline{v_{\xi}(x)}dx.\label{idtr1a}\eq

The inequality \eqref{disteig2} is an immediate consequence of \eqref{disteig}. The identity
 \eqref{gltr1} is a consequence of \eqref{trsglb} and the Grothendieck Theorem for $\frac 23$-nuclear operators (see \eqref{lia1}).
\end{proof}
We observe that the equation relating $s, r, p$ as $\,\frac{1}{s}=\frac{1}{r}-|\frac 12-\frac{1}{p}|$ for \eqref{disteig2} corresponds in the case
 $s=1$ to $\,\frac{1}{r}=1+|\frac 12-\frac{1}{p}|$. In such case the series of eigenvalues converges absolutely. In general this property does not guarantee that the Grothendieck-Lidskii formula holds. However, recently in \cite{Reinov} Latif and Reinov have proved that if $r$ and $p$ are related
 by  $\frac{1}{r}=1+|\frac 12-\frac{1}{p}|$, then the Grothendieck-Lidskii formula holds. Of course, the relevant situation is when $r$ moves along the interval $[\frac 23, 1]$. If $r\in (\frac 23,1)$  there exist two corresponding values of $p$ solving 
the equation $\frac{1}{r}=1+|\frac 12-\frac{1}{p}|$ the first one with $p < 2$ and the other one with $p > 2$. Additionaly, we can incorporate the symbol in relation with the Grothendieck-Lidskii formula as in \eqref{gltr1}. Summarising we have:

\begin{cor}\label{rlgltr} Let $0<r\leq 1$ and $1\leq p<\infty$ such that $\frac{1}{r}=1+|\frac 12-\frac{1}{p}|$. Let
$A:\cml\rightarrow \cml $ be a continuous linear operator with $L$-symbol $\sigma_A$ such that
\beq\label{inc5nb}\sum\limits_{\xi\in\mI}\|\sigma_A(\cdot,\xi)u_{\xi}(\cdot)\|_{L^{p}}^r\|v_{\xi}\|_{L^{q}}^r<\infty.\eq
Then $A$ is $r$-nuclear from $L^{p}(M)$ to $L^{p}(M)$ and the equality \eqref{gltr1} holds.
\end{cor}

\begin{proof} The $r$-nuclearity follows from Theorem \eqref{sp1}. The identity follows from the aforementioned main theorem of \cite{Reinov} and \eqref{trsglb}.
\end{proof}

We also record the following condition for operators
 with $x$-independent $L$-symbols (i.e. for $L$-Fourier multipliers).
 Indeed, if the symbol does not depend on $x$, we have,
 \[\|\sigma_A(\xi)u_{\xi}(\cdot)\|_{L^{p_2}}^r=|\sigma_A(\xi)|^r\|u_{\xi}\|_{L^{p_2}}^r\]
so that an application of Theorem \ref{sp1} imply:
 
\begin{cor}\label{ind1a} 
Let $0<r\leq 1$ and $1\leq p_1,p_2 <\infty$ with $q_1$ such that
$\frac{1}{p_1}+\frac{1}{q_1}=1$. Let $A:\cml\rightarrow\cml$ be a $L$-Fourier multiplier with $L$-symbol $\sigma_A$ such that
\beq\label{inc5a}\sum\limits_{\xi\in\mI}|\sigma_A(\xi)|^r\|u_{\xi}\|_{L^{p_2}}^r\|v_{\xi}\|_{L^{q_1}}^r<\infty.\eq
Then $A$ is $r$-nuclear from $L^{p_1}(M)$ to $L^{p_2}(M)$. Moreover, 
if $p=p_1=p_2$, then
\beq\Tr(A)=\sum\limits_{\xi\in\mI}\sigma_{A}(\xi)=\sum_{j=1}^\infty\lambda_j (A),\label{idtr1a12}\eq
where the sum of the eigenvalues $\lambda_j (A)$ of $A$ is made taking multiplicities into account. 
\end{cor}

\begin{proof} The $r$-nuclearity follows from Theorem \ref{sp1} as well as parts (i) and (ii), taking into account
 that \[\Tr(A)=\sum\limits_{\xi\in\mI}\sigma_{A}(\xi)\int\limits_{M}u_{\xi}(x)\overline{v_{\xi}(x)}dx=\sum\limits_{\xi\in\mI}\sigma_{A}(\xi),\]
 by the biorthogonality assumption. 
 The second equality in \eqref{idtr1a12} holds in view of the following Remark
 \ref{REM:invev}.
\end{proof}

\begin{rem}\label{REM:invev}
We note from formula \eqref{idtr1a12} that it holds for nuclear operators ($r=1$) on any $L^p$-space, $1\leq p<\infty$. This is not the case in general for non-invariant operators (that is, the operators which are not $L$-Fourier multipliers)  in  \eqref{gltr1} and in Corollary \ref{ind1a}  where we assumed the relation $\frac{1}{r}=1+|\frac 12-\frac{1}{p}|$ to hold. This is due to the fact that since $\sigma_A$ does not depend on $x$, from formula \eqref{sy1a} we actually have that $\sigma_A(\xi)$ is the eigenvalue of the compact (since it is nuclear) operator $A$ with the eigenfunction $u_\xi$, see also Remark \ref{REM1}. Therefore, the second equality in \eqref{idtr1a12} always holds.
\end{rem}

We now consider the special case of Schatten classes. 
Since in our setting the eigenfunctions of $L$ are not necessarily orthogonal, it is convenient for us to take advantage of the notion of $r$-nuclearity in Banach spaces instead of the usual properties characterising Schatten classes in terms of orthonormal bases. First, we establish a couple of preliminary properties.

 The following proposition is the corresponding Parseval identity for the $L$-Fourier transform. 

\begin{prop}\label{PROP:Parceval}
 Let $f,g\in L^2(M)$. Then $\widehat{f}, \widehat{g}\in \ell_{L}^2$ and 
\begin{equation}\label{EQ:Parseval}
(f,g)_{L^2}=\sum\limits_{\xi\in\mI}\widehat{f}(\xi)\overline{\widehat{g_{\ast}}(\eta)} =:(\widehat{f},\widehat{g})_{\ell_L^2},
\end{equation}
the latter inner product defining a Hilbert space ${\ell_L^2(\mI)}$.
\end{prop}
\begin{proof} The fact that $\widehat{f}, \widehat{g}\in \ell_{L}^2$ follows similar to Proposition 6.1 from \cite{rto:nhs}.  
By the Fourier inversion formula \eqref{eq:finv} and \eqref{eq:finv3}, we have
\begin{align*} (f,g)_{L^2}=&(\sum\limits_{\xi\in\mI}\widehat{f}(\xi)u_{\xi}, \sum\limits_{\xi\in\mI}\widehat{g_{\ast}}(\eta)v_{\eta})\\
=&\sum\limits_{\xi\in\mI}\widehat{f}(\xi)\overline{\widehat{g_{\ast}}(\xi)}=(\widehat{f},\widehat{g})_{\ell_L^2},
\end{align*}
and the proof is complete.
\end{proof}

\begin{prop}\label{admu} 
If $A$ is a $L$-Fourier multiplier by $\sigma_A(\xi)$, then $A^*$ is a $L^*$-Fourier multiplier by $\overline{\sigma_A(\xi)}$.
\end{prop}
\begin{proof} 
First by \eqref{EQ:Parseval} we write
\begin{equation*} (Af,g)_{L^2}=\sum\limits_{\xi\in\mI}\widehat{Af}(\xi)\overline{\widehat{g_{\ast}}(\xi)}
=\sum\limits_{\xi\in\mI}\sigma_A(\xi)\widehat{f}(\xi)\overline{\widehat{g_{\ast}}(\xi)}
=\sum\limits_{\xi\in\mI}\widehat{f}(\xi)\overline{\overline{\sigma_A(\xi)}\widehat{g_{\ast}}(\xi)}.
\end{equation*}
On the other hand
\[
(Af,g)_{L^2}=(f,A^*g)_{L^2}=\sum\limits_{\xi\in\mI}\widehat{f}(\xi)\overline{\widehat{A^*g_{\ast}}(\xi)}.\]
Therefore
\[
\widehat{A^*g_{\ast}}(\xi)=\overline{\sigma_A(\xi)}\widehat{g_{\ast}}(\xi),
\]
i.e. $A^*$ is a $L^*$-Fourier multiplier by  $\overline{\sigma_A(\xi)}$.
\end{proof}

In the following we will be looking at the membership of operators in the Schatten classes on $L^2(M)$.
Conditions for the $L^2(M)$-boundedness of operators in terms of their global symbols have been obtained
in \cite{rto:nhs}. In the case of $L$-Fourier multipliers, these conditions simplify, and the boundedness of the $L$-symbol
if enough, namely, if $\sup_{\xi\in\mI}|\sigma_A(\xi)|<\infty$ then $A$ is bounded on $L^2(M)$.

As a consequence of the preceding nuclearity considerations, we now give criteria for operators to belong
to the Schatten classes $S_r(L^2(M))$ and we refer to Remark \ref{REM:Schattens} for a 
further discussion.

\begin{cor}\label{t1} 
Let $A:\cml\rightarrow\cml$ 
be a $L$-Fourier multiplier with $L$-symbol $\sigma_A$. Then we have the following properties:
\begin{itemize}
\item[(i)] If  $0<r\leq 1$ and 
\beq\label{inc1}
\sum\limits_{\xi\in\mI}|\sigma_A(\xi)|^r<\infty,
\eq
then $A$ belongs to the Schatten class $S_r(L^2(M))$.\\
\item[(ii)] If $L_M$ is self-adjoint and $0<r<\infty$, then $A\in S_r(L^2(M))$  if and only if \eqref{inc1} holds.
\end{itemize}
\end{cor}
\begin{proof} (i) We will prove that under condition \eqref{inc1} the operator $A$ is $r$-nuclear on $L^2(M)$. Then the result will follow from the Oloff's equivalence \eqref{olo1} that holds in the setting of Hilbert spaces.

\medskip
By Corollary \ref{ind1a} applied to the case $p_1=p_2=2$ 
 and the fact that $\|u_{\xi}\|_{L^2}=\|v_{\xi}\|_{L^2}=1$, 
we obtain 
\beq\label{inc7b}\sum\limits_{\xi\in\mI}|\sigma_A(\xi)|^r\|u_{\xi}\|_{L^{2}}^r\|v_{\xi}\|_{L^{2}}^r=\sum\limits_{\xi\in\mI}|\sigma_A(\xi)|^r<\infty.\eq
Hence, $A$ is $r$-nuclear from $L^2(M)$ into $L^2(M)$. 

\medskip
(ii) Now, if $L_M$ is self-adjoint and $0<r<\infty$, we first observe that the system $\{u_{\xi}\}$ is orthonormal. For the case $r=2$ of Hilbert-Schmidt operators
 we note that, by the Plancherel identity in Proposition \ref{PROP:Parceval}
 and the well known characterisation of Hilbert-Schmidt class in terms of orthonormal bases, we have
\begin{align*} \|A\|_{S_2}^2=&\sum\limits_{\xi\in\mI}\|Au_{\xi}\|_{L^2}^2=\sum\limits_{\xi\in\mI}\|\efel (Au_{\xi})\|_{\ell_L^2}^2\\
=&\sum\limits_{\xi\in\mI}|\sigma_A(\xi)|^2\|\efel (u_{\xi})\|_{\ell_L^2}^2=\sum\limits_{\xi\in\mI}|\sigma_A(\xi)|^2\|u_{\xi}\|_{L^2}^2\\
=&\sum\limits_{\xi\in\mI}|\sigma_A(\xi)|^2.
\end{align*}
On the other hand, since $L$ is self-adjoint and by Proposition \ref{admu}, $A^*$ is a $L^{*}$-Fourier multiplier by $\overline{\sigma_A(\xi)}$.
 Hence $A^*A$ is also a $L$-Fourier multiplier with symbol $\sigma_{A^*A}(\xi)=|\sigma_A(\xi)|^2$. Since $A^*A$ is a positive operator, more generally we have
$\sigma_{|A|^{2s}}(\xi)=|\sigma_A(\xi)|^{2s}$ for all $s\in\er$.

\medskip 
Therefore  
\begin{equation*} \|A\|_{S_r}^r= \||A|\|_{S_r}^r=\||A|^{\frac{r}{2}}\|_{S_2}^2
=\sum\limits_{\xi\in\mI}\sigma_{|A|^{\frac{r}{2}}}(\xi)^2=\sum\limits_{\xi\in\mI}|\sigma_A(\xi)|^r,
\end{equation*}
completing the proof.
\end{proof}

\begin{rem} \label{REM:Schattens}
When studying Schatten classes for multipliers  directly from the definition, the problem of 
 understanding the composition $A^*A$ arises. Even in the simplest case of a $L$-Fourier multiplier $A$, we note the difficulty since $A^*$ is a $L^*$-Fourier multiplier, but not necessarily a $L$-Fourier multiplier. 
 This observation explains why the method of studying Schatten classes
 via the notion of $r$-nuclearity on Banach spaces is more appropriate in this case. 
 The Corollary \ref{t1} gives a taste of it. The difference between two parts of Corollary \ref{t1} is that if $L_M$ is not
 self-adjoint, the operator $AA^*$ is not an $L$-Fourier multiplier, and its symbol involves a more
 complicated symbolic calculus, as developed for general operators in \cite{rto:nhs}.
\end{rem}

 We now consider an example in the case of $L$-Fourier multipliers. According to Definition \ref{Lfm}, it is clear
 that $L$ itself is a $L$-Fourier multiplier. Indeed, if $f\in\cml$, by \eqref{eq:finv} we obtain
\[Lf(x)=\sum\limits_{\xi\in\mI}Lu_{\xi}(x)\widehat{f}(\xi)=\sum\limits_{\xi\in\mI}\lambda_{\xi}u_{\xi}(x)\widehat{f}(\xi).\]
Hence $\sigma_{L}(\xi)=\lambda_{\xi}$.\\

Now, if $-L$ is a positive operator, i.e. if all eigenvalues satisfy $-\lambda_\xi\geq 0$ for all $\xi\in\mI$
(thinking of an example $L=\Delta$ the Laplacian), the operator
$I-L$ is strictly positive and we can define $(I-L)^{-s}$ for every real $s>0$. The operator 
 $(I-L)^{-\frac{s}{m}}$ is a $L$-Fourier multiplier and  $\sigma_{(I-L)^{-\frac{s}{m}}}(\xi)=
 (1-\lambda_{\xi})^{-\frac{s}{m}}\cong\langle\xi\rangle^{-s}$.\\

Hence, by Corollary \ref{t1} (i) and under Assumption \eqref{we2} on the exponent $s_0$, we obtain: 

\begin{cor}\label{cs1a} 
Assume that $-L$ is positive and 
let $0<r\leq 1$. Then
$(I-L)^{-\frac{s}{m}}\in S_r(L^2(M))$ for $s>\frac{s_0}{r}.$
\end{cor}

We now derive some consequences when bounds on the biorthogonal system are avalaible.
We will assume that there exist constants $C_1,C_2, \nu_1, \nu_2>0$ such that
\beq\label{con1a} 
\|u_{\xi}\|_{L^{\infty}(M)}\leq C_1\langle\xi\rangle^{\nu_1}
\quad\textrm{ and }\quad
\|v_{\xi}\|_{L^{\infty}(M)}\leq C_2\langle\xi\rangle^{\nu_2}.
\eq
We will also require the following lemma.
\begin{lem} \label{LEM:ineq} For the biorthogonal system $u_{\xi}, v_{\xi}$ associated to the operator $L$ and satisfying \eqref{con1a} we have
\begin{equation}\label{ineqA}
 \|u_{\xi}\|_{L^{q}(M)} \leq\left\{
\begin{array}{rl}
(C_1\langle\xi\rangle^{\nu_1})^{1-\frac{2}{q}} ,& 2\leq q\leq \infty,\\
|M|^{\frac{1}{q}-\frac{1}{2}} , & 1\leq q\leq 2.
\end{array} \right.
\end{equation}

\medskip
The analogous inequalities hold for $\|v_{\xi}\|_{L^{q}(M)} $ with $C_2$ and $\nu_2$ instead
 of $C_1$ and $\nu_1$, respectively, in the above inequalities, namely,
 \begin{equation}\label{ineqA2}
 \|v_{\xi}\|_{L^{q}(M)} \leq\left\{
\begin{array}{rl}
(C_2\langle\xi\rangle^{\nu_2})^{1-\frac{2}{q}} ,& 2\leq q\leq \infty,\\
|M|^{\frac{1}{q}-\frac{1}{2}} , & 1\leq q\leq 2.
\end{array} \right.
\end{equation}
\end{lem}
\begin{proof} 
If $q=\infty$, the inequality \eqref{ineqA} is the same as \eqref{con1a}.
 If $2\leq q<\infty$ we apply the inequality 
$$\|f\|_{L^{q}}\leq \|f\|_{L^{\infty}}^{\frac{q-2}{q}}\|f\|_{L^{2}}^{\frac{2}{q}}.$$
Then
\[
\|u_{\xi}\|_{L^{q}(M)} \leq \|u_{\xi}\|_{L^{\infty}}^{\frac{q-2}{q}}\leq (C_1\langle\xi\rangle^{\nu_1})^{1-\frac{2}{q}} .
\]
Finally, for $1\leq q\leq 2$, using H\"older's inequality, we get 
 \[\|u_{\xi}\|_{L^{q}(M)}^{q} =  \int\limits_M|u_{\xi}(y)|^{q}dy \leq   \left(\int\limits_M1dy\right)^{1-\frac{q}{2}} \left(\int\limits_M|u_{\xi}(y)|^{q\frac{2}{q}}dy\right)^{\frac{q}{2}}
\leq  |M|^{1-\frac{q}{2}}.
\]
This implies \eqref{ineqA}, with the proof of \eqref{ineqA2} being completely analogous.
\end{proof}

Under the assumption of \eqref{con1a} for the biorthogonal system we have:
\begin{thm} \label{sevc}Let $0<r\leq 1$ and let $A:\cml\rightarrow \cml$ be a continuous linear operator with $L$-symbol $\sigma_A$ such that
$|\sigma(x,\xi)|\leq \gamma(\xi)$ for all $(x,\xi)$ and some function $\gamma:\mI\rightarrow [0,\infty)$. Then 
we have the following properties:
\begin{enumerate}
\item[(i)] if $1\leq p_1\leq 2,\,\, 2\leq p_2<\infty$, and 
\[\sum\limits_{\xi\in\mI}(\langle\xi\rangle^{\nu_1(1-\frac{2}{p_2})+\nu_2(\frac{2}{p_1}-1)}\gamma(\xi))^r<\infty,\]
then $A$ is $r$-nuclear from $L^{p_1}(M)$ to $L^{p_2}(M)$.
\item[(ii)] if $2\leq p_2<\infty$ and 
\[\sum\limits_{\xi\in\mI}(\langle\xi\rangle^{\nu_1(1-\frac{2}{p_2})}\gamma(\xi))^r<\infty,\]
then $A$ is $r$-nuclear from $L^{p_1}(M)$ to $L^{p_2}(M)$ for all $2\leq p_1<\infty$.
\item[(iii)] $1\leq p_1\leq 2$ and 
\[\sum\limits_{\xi\in\mI}(\langle\xi\rangle^{\nu_2(\frac{2}{p_1}-1)}\gamma(\xi))^r<\infty,\]
then $A$ is $r$-nuclear from $L^{p_1}(M)$ to $L^{p_2}(M)$ for all $1\leq p_2\leq 2$.
\item[(iv)]  if
\[\sum\limits_{\xi\in\mI}\gamma(\xi)^r<\infty,\]
then $A$ is $r$-nuclear from $L^{p_1}(M)$ to $L^{p_2}(M)$ for all $2\leq p_1<\infty$ and all $1\leq p_2\leq 2$.
\end{enumerate}
\end{thm}
\begin{proof} 
We write $g_{\xi}(x):=u_{\xi}(x)\sigma_A(x,\xi),\,\, h_{\xi}(y):=\overline{v_{\xi}(y)}$ for the decomposition of the kernel of $A$ as in the proof of Theorem \ref{sp1}. We are just going to prove (i), the other statements can be argued in a similar way. Let us fix 
  $q_1$ such that
$\frac{1}{p_1}+\frac{1}{q_1}=1$.

We note that, by using  \eqref{con1a} and Lemma \ref{LEM:ineq} for $p_2, q_1$ respectively we obtain
\[\|g_{\xi}\|_{L^{p_2}}^r\|h_{\xi}\|_{L^{q_1}}^r\leq (\gamma(\xi)\langle\xi\rangle^{\nu_1(\frac{p_2-2}{p_2})+\nu_2(\frac{q_1-2}{q_1})})^r.\]
Since $\frac{q_1-2}{q_1}=1-\frac{2}{q_1}=1-2(1-\frac{1}{p_1})=\frac{2}{p_1}-1$, the $r$-nuclearity of $A$ then follows from Theorem \ref{sp1}.
\end{proof}
A case where the situation above arises is when we dispose of a bound for the symbol $\sigma_A$
 of the type:
 \beq\label{syc}|\sigma_A(x,\xi)|\leq C\langle\xi\rangle^{-s},\eq
for some suitable positive constants $C,s$. 

 In the corollary below we consider the case $p=p_1=p_2$, in which case we have:
\begin{cor}\label{sevc2} Let $0<r\leq 1$ and let $A:\cml\rightarrow\cml$ be a continuous linear operator with $L$-symbol $\sigma_A$ such that \eqref{con1a} and \eqref{syc} hold. 
Then we have the following properties.
\begin{enumerate} 
\item[(i)] If $2\leq p<\infty$ and $s\geq \nu_1(\frac{p-2}{p})+\frac{s_0}{r}$, so that
\[\sum\limits_{\xi\in\mI}\langle\xi\rangle^{(\nu_1(\frac{p-2}{p})-s)r}<\infty,\]
then $A$ is $r$-nuclear from $L^{p}(M)$ into $L^{p}(M)$.
\item[(ii)] If $1\leq p\leq 2$ and $s\geq \nu_2(\frac{2}{p}-1)+\frac{s_0}{r}$, so that 
\[\sum\limits_{\xi\in\mI}\langle\xi\rangle^{(\nu_2(\frac{2}{p}-1)-s)r}<\infty,\]
then $A$ is $r$-nuclear from $L^{p}(M)$ to $L^{p}(M)$.
\item[(iii)] If additionally in $(i)$ or $(ii)$ we have $r\leq \frac 23$ then \eqref{gltr1} holds.  
\item[(iv)] If additionally in $(i)$ or $(ii)$ we have  $\frac{1}{r}=1+|\frac 12-\frac{1}{p}|$, then \eqref{gltr1} holds.
\end{enumerate}
\end{cor}
\begin{proof} (i) We set $q$ such that $\frac{1}{p}+\frac{1}{q}=1$. We note that for $\gamma(\xi)=C\langle\xi\rangle^{-s}$ with
 $C, s$ as in \eqref{syc} and taking into account \eqref{we2} we obtain 
\[\sum\limits_{\xi\in\mI}(\langle\xi\rangle^{\nu_1(1-\frac{2}{p})}\gamma(\xi))^r=C^r\sum\limits_{\xi\in\mI}\langle\xi\rangle^{(\nu_1(\frac{p-2}{p})-s)r}<\infty\]
provided that $\nu_1(\frac{p-2}{p})r-sr\leq -s_0$. The $r$-nuclearity of $A$ follows now from part (ii) of Theorem \ref{sevc}.

The proof of (ii) can be deduced from (iii) of Theorem \ref{sevc} in a similar way.

 The part (iii) follows analogously as for \eqref{gltr1} and (iv) follows analogously to the corresponding argument for Corollary \ref{rlgltr}.
\end{proof}

\section{Applications and examples}
\label{SEC:example}

In this section we consider special cases of boundary value problems $L_{M}$, for the  
manifold  $M=[0,1]^n$.  
These provide some examples of problems where our method is applicable and show how
to apply it in similar settings.

\subsection{Non-periodic boundary conditions}
\label{SEC:nonper}

 We start with the case of $L={\rm O}_h^{(n)}$ which we now define.
 The operator ${\rm O}_h^{(n)}$ is a natural extension of the one-dimensional operator ${\rm O}_h^{(1)}$ 
 considered in detail in 
 \cite{rto:nhs,Kanguzhin_Tokmagambetov_Tulenov}, 
 and an extension of the setting of periodic operators to the non-periodic setting. 

\medskip
We first formulate a characterisation of Hilbert-Schmidt operators in terms of the ${\rm O}_h^{(n)}$-symbols. For the sake of clarity 
 we will write $L_h$ instead of ${\rm O}_h^{(n)}$. We briefly recall the basic facts about $L_h$. 

We set $M=\overline{\Omega}$ for $\Omega=(0,1)^n$ and 
\[ \arn_+:=\{h=(h_1,\dots,h_n)\in\arn : h_j>0 \mbox{ for every } j=1,\dots,n\}.\] 
 For $h\in\arn_+$, the operator $L_h$ on $\Omega$ is defined by 
\beq\label{lae1} L_h =\Delta=\sum_{j=1}^{n}\frac{\partial^2}{\partial x_j^2},\eq
together with the boundary conditions (BC):
\begin{equation}\label{EQ:BCh}
h_{j} f(x)|_{x_{j}=0}=f(x)|_{x_{j}=1},\quad
h_{j} \frac{\partial f}{\partial x_j}(x)|_{x_{j}=0}=\frac{\partial f}{\partial x_j}(x)|_{x_{j}=1},
\quad j=1,\ldots,n,
\end{equation}
and the domain
\[D(L_{h})=\{f\in L^{2}(\Omega) : \Delta f\in L^{2}(\Omega):\;
f \textrm{ satisfies \eqref{EQ:BCh}} \}.\]

In order to describe the corresponding biorthogonal system, we first note that since $b^0=1$ for all $b>0$, we can define $0^0=1$. In particular
 we write 
 $$h^x=h_1^{x_1}\cdots h_n^{x_n}=\prod\limits_{j=1}^{n}h_{j}^{x_{j}}$$ 
 for $x\in [0,1]^n$. 
Then, the system of
eigenfunctions of the operator $L_h$ is
$$\{u_{\xi}(x)=h^{x}\exp(i 2\pi\xi x ), \xi\in
\mathbb{Z}^{n}\}$$ 
and the conjugate system is
$$\{v_{\xi}(x)=h^{-x}\exp(i 2\pi\xi x ),\,\, \xi\in
\mathbb{Z}^{n}\},$$ 
where $\xi x=\xi_{1}x_{1}+ ... + \xi_{n}x_{n}$. Note that
$u_{\xi}(x)=\prod\limits_{j=1}^{n}u_{\xi_{j}}(x_{j}),$ where
$u_{\xi_{j}}(x_{j})=h_{j}^{x_{j}}\exp(i 2\pi \xi_{j} x_{j})$.\\
 
We can now write an operator $A$ on $C^\infty_{L_h}(\Omega)$ in terms of its $L_h$-symbol in the following way
\[
Af(x)=\sum\limits_{\xi\in\Zn}\int\limits_{\Omega}u_{\xi}(x)\overline{v_{\xi}(y)}\sigma_A(x,\xi)f(y)dy,
\]
where $u_{\xi}(x)=h^{x}\exp(i2\pi\xi x),\, v_{\xi}(y)=h^{-y}\exp(i2\pi\xi y)$, 
and where {\em we have renumbered $\xi$ taking it in $\Zn$ instead of $\ene_0$}. 
Of course, such a renumbering does not change any essential properties of operators and their symbols, but is
more in resemblance of the toroidal analysis in \cite{Ruzhansky-Turunen-JFAA-torus}.
We denote 
 $$h_j^{(\min)}:=\min\{1, h_j\}, \quad h_j^{(\max)}:=\max\{1, h_j\},$$ 
 for every $j=1,\dots,n$.  
 
 Unless $h=1$, the problem $L_\Omega$ and its boundary conditions (BC) are not self-adjoint.
 In particular, it means that in general, we can not consider compositions like $AA^*$ on the domains of
 these operators. However, we can note that the spaces $C_L^\infty(\Omega)$ and $C_{L^*}^\infty(\Omega)$
 are dense in $L^2(\Omega)$, and the usual test function space
 $C^\infty_0(\Omega)$ is dense in these test-functions as well.
 In order to avoid such technicalities at this point, in theorem below we can define the operator
Hilbert-Schmidt norm $\|A\|_\HS$ as $(\int_\Omega\int_\Omega |K_A(x,y)|^2 dx dy)^{1/2}$, where
$K_A(x,y)$ is the Schwartz integral kernel of the operator $A$. 
This may be viewed as a natural extension of the well-known property for problems without boundary conditions.
Otherwise, it is not restrictive to assume that $A$ is a bounded compact operator on $L^2(\Omega)$ in which
case such questions do not arise.

The purpose of the following statement is 
to express the membership of an operator in the Hilbert-Schmidt class in terms of its global symbols, and to
emphasise the dependence of this norm on the parameter $h$ entering the boundary conditions.
 
\begin{thm}\label{mainhs} 
Let $h\in \arn_+$.
Let $A:C_{L_h}^{\infty}(\Omega)\rightarrow C_{L_h}^{\infty}(\Omega)$ 
be a continuous linear operator with $L_h$-symbol $\sigma_A$. Then $A$ is a Hilbert-Schmidt operator if and only if 
\begin{equation}\label{EQ:HS-cond}
\int\limits_{\Omega}\sum\limits_{\xi\in\Zn}|\sigma_A(x,\xi)|^2dx<\infty.
\end{equation}
Moreover,
\[\|A\|_{\HS}=\left(\int\limits_{\Omega}\sum\limits_{\xi\in\Zn}|(\mathcal{F}_{\Tn}h^{(\cdot)})\ast \sigma_A(x,\cdot)(\xi)|^2dx\right)^{\half},\] 
where 
\beq
(\mathcal{F}_{\Tn}h^{(\cdot)})\ast \sigma_A(x,\cdot)(\xi)=\sum\limits_{\eta\in\Zn}(\mathcal{F}_{\Tn}h^{(\cdot)})(\xi-\eta) \sigma_A(x,\eta)\label{hsa1}
\eq
and
$(\mathcal{F}_{\Tn}h^{(\cdot)})(\xi)=\int\limits_{\Omega} e^{-i2\pi\xi z} h^z dz$.
In particular,
\[C_{\min}(h)
\left(\int\limits_{\Omega}\sum\limits_{\xi\in\Zn}|\sigma_A(x,\xi)|^2dx\right)^{1/2}
\leq\|A\|_{\HS}\leq C_{\max}(h)
\left(\int\limits_{\Omega}\sum\limits_{\xi\in\Zn}|\sigma_A(x,\xi)|^2dx\right)^{1/2},\]
where $C_{\min}(h)=\prod\limits_{j=1}^nh_j^{(\min)}$ and $C_{\max}(h)=\prod\limits_{j=1}^nh_j^{(\max)}.$

\end{thm}
\begin{proof}
First we observe that the kernel $K_A(x,y)$ can be written in the form
\begin{align*}
K_A(x,y)=&\sum\limits_{\xi\in\Zn}u_{\xi}(x)\overline{v_{\xi}(y)}\sigma_A(x,\xi)\\
=&\sum\limits_{\xi\in\Zn}h^{x}\exp(i2\pi\xi x)h^{-y}\exp(-i2\pi\xi y)\sigma_A(x,\xi)\\
=&h^{x-y}\sum\limits_{\xi\in\Zn}e^{i2\pi\xi(x-y)}\sigma_A(x,\xi).
\end{align*}
Hence
\begin{equation*}
K_A(x,x-z)=h^{z}\sum\limits_{\xi\in\Zn}e^{i2\pi\xi z}\sigma_A(x,\xi)
=h^{z}\mathcal{F}_{\Tn}^{-1}\sigma_A(x,\cdot)(z).
\end{equation*}
We also note that
\begin{align*}
(\mathcal{F}_{\Tn}h^{(\cdot)})\ast \sigma_A(x,\cdot)(\xi)=&\sum\limits_{\eta\in\Zn}(\mathcal{F}_{\Tn}h^{(\cdot)})(\eta) \sigma_A(x,\xi-\eta)\\
=&\sum\limits_{\eta\in\Zn}(\mathcal{F}_{\Tn}h^{(\cdot)})(\xi-\eta) \sigma_A(x,\eta).
\end{align*}
Therefore 
\begin{align}
\int\limits_{\Omega}\int\limits_{\Omega}|K_A(x,x-z)|^2dxdz&=\int\limits_{\Omega}\int\limits_{\Omega}|h^{z}\mathcal{F}_{\Tn}^{-1}\sigma_A(x,\cdot)(z)|^2dxdz\label{comp1ah}\\
&=\int\limits_{\Omega}\int\limits_{\Omega}|\mathcal{F}_{\Tn}^{-1}((\mathcal{F}_{\Tn}h^{(\cdot)})\ast \sigma_A(x,\cdot))(z)|^2dxdz\nonumber\\
&=\int\limits_{\Omega}\sum\limits_{\xi\in\Zn}|(\mathcal{F}_{\Tn}h^{(\cdot)})\ast \sigma_A(x,\cdot)(\xi)|^2dx.\nonumber
\end{align}
Then, the Hilbert-Schmidt norm of A is given by 
\[\|A\|_{\HS}^2=\int\limits_{\Omega}\sum\limits_{\xi\in\Zn}|(\mathcal{F}_{\Tn}h^{(\cdot)})\ast \sigma_A(x,\cdot)(\xi)|^2dx,
\] 
with $(\mathcal{F}_{\Tn}h^{(\cdot)})\ast \sigma_A(x,\cdot)(\xi)=\sum\limits_{\eta\in\Zn}(\mathcal{F}_{\Tn}h^{(\cdot)})(\xi-\eta) \sigma_A(x,\eta).$

On the other hand, since the function $h^x:[0,1]^n\rightarrow \er$ is continuous, the numbers 
 $C_{\min}=\prod\limits_{j=1}^nh_j^{(\min)}, C_{\max}=\prod\limits_{j=1}^nh_j^{(\max)}$ are well defined. 
 By \eqref{comp1ah} we obtain
\[
C_{\min}^2\int\limits_{\Omega}\int\limits_{\Omega}|\mathcal{F}_{\Tn}^{-1}\sigma_A(x,\cdot)(z)|^2dxdz
\leq\|A\|_{\HS}^2
\leq C_{\max}^2\int\limits_{\Omega}\int\limits_{\Omega}|\mathcal{F}_{\Tn}^{-1}\sigma_A(x,\cdot)(z)|^2dxdz.
\]
By the usual Plancherel identity on $\Tn$ we have
\[\int\limits_{\Omega}\sum\limits_{\xi\in\Zn}|\sigma_A(x,\xi)|^2dx=\int\limits_{\Omega}\int\limits_{\Omega}|\mathcal{F}_{\Tn}^{-1}\sigma_A(x,\cdot)(z)|^2dxdz,\]
which concludes the proof.
\end{proof}
\begin{rem} We note that if $h=1$ we have $(\mathcal{F}_{\Tn}h^{(\cdot)})(\xi)=\delta(\xi)$ 
where $\delta$ is the Dirac delta on $\Zn$. Then
\[(\mathcal{F}_{\Tn}h^{(\cdot)})\ast \sigma_A(x,\cdot)(\xi)=\sigma_A(x,\xi).\]
Thus the theorem above recovers the well known characterisation of Hilbert-Schmidt operators in terms of the square 
integrability of the symbol. Again, for $h=1$ the following results fall in the framework of the nuclearity properties of
operators on compact Lie groups which have been analysed in \cite{dr13a:nuclp}. The pseudo-differential calculus in
this case coincides with the toroidal pseudo-differential calculus of operators on the tori developed in
\cite{Ruzhansky-Turunen-JFAA-torus}, see also \cite{rt:book}.
\end{rem}

We shall now establish some results in relation with the nuclearity. 
\begin{thm}\label{rnuca} 
Let $0<r\leq 1$, $1\leq p_1,p_2 <\infty$, $h\in\arn_+$, and let $A:C_{L_{h}}^{\infty}(\Omega)\rightarrow C_{L_{h}}^{\infty}(\Omega) $ be a continuous linear operator with $L_{h}$-symbol $\sigma_A$.  If
\[\sum\limits_{\xi\in\Zn}\|\sigma_A(\cdot,\xi)\|_{L^{p_2}(\Omega)}^r<\infty,\]
then  $A$ is $r$-nuclear from $L^{p_1}$ to $L^{p_2}$ for all $p_1$ with $1\leq p_1 <\infty$.  
 Moreover, the $r$-nuclear quasi-norm satisfies
\[n_r^r(A)\leq C_{h}^r\sum\limits_{\xi\in\Zn}\|\sigma_A(\cdot,\xi)\|_{L^{p_2}(\Omega)}^r ,\]
where $C_{h}$ is a positive constant which only depends on $h$. 
\end{thm}
\begin{proof} From the proof of Theorem \ref{mainhs} we use the formula for the kernel $K_A(x,y)$ of $A$ to get
\begin{align*}
K_A(x,y)=&\sum\limits_{\xi\in\Zn}h^{x}\exp(i2\pi\xi x)h^{-y}\exp(-i2\pi\xi y)\sigma_A(x,\xi)\\
=&\sum\limits_{\xi\in\Zn}\alpha_{\xi}(x)\beta_{\xi}(y),
\end{align*}
where $\alpha_{\xi}(x)=h^{x}\exp(i2\pi\xi x)\sigma_A(x,\xi)$ and $\beta_{\xi}(y)=h^{-y}\exp(-i2\pi\xi y)$.\\

We put $q_1$ such that $\frac{1}{p_1}+\frac{1}{q_1}=1$ and observe that  
\begin{equation*} \|\alpha_{\xi}\|_{L^{p_2}}^r\leq \left(\int\limits_{\Omega}|h^x|^{p_2}|\sigma_A(x,\xi)|^{p_2}dx\right)^{\frac{r}{p_2}}
\leq (C_{h }')^r\left(\int\limits_{\Omega}|\sigma_A(x,\xi)|^{p_2}dx\right)^{\frac{r}{p_2}},
\end{equation*}
where $C_{h}'=\max\{1, h_{1},h_{2},\cdots,h_{n} \}$. On the other hand
\begin{equation*} \|\beta_{\xi}\|_{L^{q_1}}^r\leq \left(\int\limits_{\Omega}|h^{-y}|^{q_1}dx\right)^{\frac{r}{q_1}}
\leq (\widetilde{C}_{h })^r,
\end{equation*}
where $\widetilde{C}_{h }=\max\{1, h_{1}^{-1},h_{2}^{-1},\cdots,h_{n}^{-1} \}$.\\

Therefore 
\[\sum\limits_{\xi\in\Zn}\|\alpha_{\xi}\|_{L^{p_2}}^r\|\beta_{\xi}\|_{L^{q_1}}^r\leq C_{h}^r\sum\limits_{\xi\in\Zn}\left(\int\limits_{\Omega}|\sigma_A(x,\xi)|^{p_2}dx\right)^{\frac{r}{p_2}}, \]
where $C_{h}=C_{h }'\widetilde{C}_{h }.$ An application of Theorem \ref{sp1} concludes the proof.  
The corresponding inequality for the quasi-norm is an immediate consequence of its definition and the estimation above.
\end{proof}
In the particular case of $p=2$ we obtain the following corollary for Schatten classes.

\begin{cor}\label{cora1a} 
Let $0<r\leq 1$, $h\in\arn_+$, and let $A:C_{L_h}^{\infty}(\overline{\Omega})\rightarrow C_{L_h}^{\infty}(\overline{\Omega}) $ be a continuous linear operator with $L_h$-symbol $\sigma_A$ such that 
\[\sum\limits_{\xi\in\Zn}\|\sigma_A(\cdot,\xi)\|_{L^{2}}^r<\infty.\]
Then  $A\in S_r(L^2)$ and the Schatten quasi-norm satisfies
\[\|A\|_{S_r}^r\leq C_{h }^r\sum\limits_{\xi\in\Zn}\|\sigma_A(\cdot,\xi)\|_{L^{2}}^r ,\]
\end{cor}

\begin{ex}
In particular, from Corollary \ref{cs1a}, with any $s_0>n$, we have 
\[(I-L_h)^{-\frac{s}{2}}\in S_r(L^2(\Omega)),\] 
provided that $sr>n,$ $0<r\leq 1$.
\end{ex}

In general for the notion of $L$-Fourier multipliers we note that, in the case of $L=L_h$ we  have:

\begin{prop} \label{PROP:Lh-inv}
Let $P(D)=\sum\limits_{|\alpha|\leq m}a_{\alpha}\partial^{\alpha}$ 
be a partial differential operator with constants coefficients on $\Omega=(0,1)^n$, then $P(D)$ is a $L_h$-Fourier multiplier
for any $h\in\mathbb R_+^n$.
\end{prop}
\begin{proof} We write $A=P(D)$. Since $\sigma_{A}(x,\xi)=u_{\xi}(x)^{-1}P(D)u_{\xi}(x)$, the $L_h$-symbol of $A$ is given by 
\[\sigma_{A}(x,\eta)=h^{-x}e^{-2\pi ix\eta}P(D)(h^xe^{2\pi ix\eta}).\]
It is not difficult to see that 
\[P(D)(h^xe^{2\pi ix\eta})=\sum\limits_{|\alpha|\leq m}a_{\alpha}b_{\alpha}(h,\eta)e^{2\pi ix\eta},\]
where $b_{\alpha}(h,\eta)$ is some polynomial in $\eta$ of degree $\leq |\alpha|$.\\

Therefore, $\sigma_{A}(\xi)=\sum\limits_{|\alpha|\leq m}a_{\alpha}b_{\alpha}(h,\xi)$ and $A=P(D)$ is a $L_h$-Fourier multiplier. 
\end{proof} 
It follows from the definition of $L$-Fourier multipliers that the class of $L$-Fourier multipliers is closed under
compositions. Thus, compositing invariant operators from Proposition \ref{PROP:Lh-inv} with e.g.
powers $(I-L_h)^{-s}$ we can obtain many examples of $L_h$-Fourier multipliers of different orders.


\subsection{Non-local boundary condition}
\label{SEC:nonloc}

Given the information on the model boundary value problem similar conclusions can be drawn for other operators.
We briefly give another example of a non-local boundary condition, see
\cite[Example 2.4]{rto:nhs} or \cite{Kanguzhin-Nurahmetov:Kaz-2002} for more details and proofs of the following
spectral properties that we now summarise.
 We now consider $M=[0,1]$ and  the 
operator $L=-i\frac{d}{dx}$ on $\oM=\Omega=(0,1)$ with the domain
$$
D({L})=\left\{f\in W_2^1[0,1]: af(0)+bf(1)+\int_{0}^1 f(x) q(x) dx=0\right\},
$$
where $a\not=0$, $b\not=0$, and $q\in C^1[0,1]$.
We assume that $a+b+\int_0^1 q(x) dx=1$ so that the inverse ${L}^{-1}$ exists and is bounded.
The operator ${\rm L}$ has a discrete spectrum and its eigenvalues can be enumerated so that
$$
\lambda_{j}=-i\ln (-\frac{a}{b})+2j\pi+\alpha_j, \ j\in \mathbb{Z},
$$
and for any $\epsilon>0$ we have $\sum_{j\in\mathbb Z} |\alpha_j|^{1+\epsilon}<\infty$.
If $m_j$ denotes the multiplicity of the eigenvalue $\lambda_j$, then $m_j=1$ for sufficiently large $|j|$.
The system of extended eigenfunctions
\begin{equation}\label{EQ:exev}
u_{jk}(x)=\frac{(ix)^k}{k!} e^{ i \lambda_j x }:\quad 0\leq k\leq m_j-1,\; j\in \mathbb{Z},
\end{equation}
of the operator $L$ is 
a Riesz basis in $L^{2}(0,1)$,
and its biorthogonal system is given by
$$
v_{jk}(x)=\lim_{\lambda\to\lambda_j} \frac{1}{k!} \frac{d^k}{d\lambda^k}
\left(\frac{(\lambda-\lambda_j)^{m_j}}{\Delta(\lambda)}
(ibe^{i\lambda(1-x)}+i\int_x^1 e^{i\lambda(t-x)}q(t) dt)
\right),
$$
$0\leq k\leq m_j-1,\;  j\in \mathbb{Z}$, where
$\Delta(\lambda)=a+b e^{i\lambda}+\int_0^1  e^{i\lambda x} q(x) dx.$ 
It can be shown that eigenfunctions $e^{i\lambda_j x}$ satisfy
\begin{equation}\label{EQ:exl2}
\sum_{j\in\mathbb Z} \|e^{i\lambda_j x}-e^{i 2\pi j x}\|^2_{L^2(0,1)}<\infty.
\end{equation}
In particular, this implies that modulo finitely many elements, the system
\eqref{EQ:exev} is a WZ-system.
Moreover, it follows that modulo terms for finitely many $j$, operators
in Proposition \ref{PROP:Lh-inv} are $L$-Fourier multipliers.
In view of the indexing notation in \eqref{EQ:exev} it is convenient to adjust accordingly the
indexing for the whole analysis. 

Now, as it was mentioned above
it is possible to take $j_0\in\mathbb N$ large enough 
so that $m_j=1$ for $|j|\geq j_0$. 
Denoting by $P_{|j|<j_0}$ and $P_{|j|\geq j_0}$ the spectral projections to
$[0,j_0)$ and $[j_0,\infty)$, respectively, 
so that $Pu_{j,k}=u_{j,k}$ for all $|j|<j_0$ and $0\leq k\leq m_j-1$, and
$Pu_{j,0}=u_{j,0}$ for all $|j|\geq j_0$,
we have
$P_{|j|<j_0}+P_{|j|\geq j_0}=I$ and we can decompose any linear 
(and suitably continuous) operator as $A=A P_{|j|<j_0}+A P_{|j|\geq j_0}$. 
Now, using the decomposition $f=\sum_{j\in\mathbb Z}\sum_{k=0}^{m_j-1} \widehat{f}(j,k) u_{j,k}$, 
with $\widehat{f}(j,k)=(f,v_{jk})_{L^2(0,1)}$,
the operator 
$A P_{|j|<j_0}f=\sum_{j=-j_0+1}^{j_0-1}\sum_{k=0}^{m_j-1} \widehat{f}(j,k) Au_{j,k}$ is a finite sum and hence
belongs to all Schatten classes and satisfies nuclearity properties of any order.
On the other hand, the operator $A P_{|j|\geq j_0}$ has simple eigenfunctions and its analysis is the same
as that carried out in Section \ref{SEC:nonper}, so we can omit the details.





\end{document}